\documentclass[12pt,a4paper,twoside]{amsart}
\usepackage{amsmath,amssymb, MnSymbol, amsthm, stmaryrd, MnSymbol, hyperref, amsfonts, bbding, pifont}
\usepackage{rotating}
\usepackage{graphicx}
\usepackage{xcolor,color}
\usepackage{enumerate}
\usepackage[utf8]{inputenc}
\usepackage[textsize=small]{todonotes}
\makeatletter
\newcommand{\namedlabel}[2]{%
  \def\@currentlabel{#1}%
  \label{#2}%
}
\makeatother
\newtheorem{thm}{Theorem}[section]
\newtheorem{cor}[thm]{Corollary}
\newtheorem{prop}[thm]{Proposition}
\newtheorem{lem}[thm]{Lemma}
\newtheorem{fact}[thm]{Fact}

\newtheorem{question}[thm]{Question}
\newtheorem*{prop*}{Proposition}
\newtheorem*{thm*}{Theorem}

\theoremstyle{definition}
\newtheorem{dfn}[thm]{Definition}
\newtheorem{exm}[thm]{Example}

\newcommand{\lbr}{\llbracket}
\newcommand{\rbr}{\rrbracket}

\newcommand\zfc{{\mathsf{ZFC}}}
\newcommand\zf{{\mathsf{ZF}}}

\usepackage{todonotes}

\newcommand\gl{\mathsf{GL}}
\newcommand\lang{\mathcal{L}}
\newcommand{\tup}[1]{\langle #1\rangle}
\newcommand{\set}[1]{\{ #1\}}

\newcommand{\inter}[1]{\llbracket#1\rrbracket}

\newcommand{\fil}{\mathcal F}

\newcommand{\model}{\mathfrak M}

\newcommand\Suc{\mathrm{Suc}}
\renewcommand\root{\mathrm{root}}
\newcommand\cof{\mathrm{cf}}

\title{Strong Completeness of Provability Logic for Uncountable Languages}

\author[M.  Golshani]{Mohammad Golshani}

\address{Mohammad Golshani, School of Mathematics, Institute for Research in Fundamental Sciences (IPM), P.O.\ Box:
	19395--5746, Tehran, Iran.}

\email{golshani.m@gmail.com}
\author[G. Stepanov]{Grigorii Stepanov}

\address{Grigorii Stepanov,
Institut f\"ur diskrete Mathematik und Geometrie, TU Wien. Wiedner Hauptstrasse 8-10, 1040 Vienna, Austria.}
\email{grigorii.stepanov@tuwien.ac.at}

\author[R. Zoghifard]{Reihane Zoghifard}

\address{Reihane Zoghifard, School of Mathematics, Institute for Research in Fundamental Sciences (IPM), P.O.\ Box:
	19395--5746, Tehran, Iran.}

\email{r.zoghi@gmail.com}

\begin{document}

\subjclass[2010]{Primary: 03F45, Secondary: 03E10, 03F15, 54G12}

\keywords {G\"odel-L\"ob provability logic; Uncountable language; Strong completeness; Topological semantics; Erd\H{o}s--Rado theorem}
	\maketitle
	
\begin{abstract}
	For an ordinal $\lambda>0$, we use the Erd\H{o}s--Rado partition theorem to prove the failure of strong completeness of $\gl$ for modal languages of cardinality  $(2^{|\lambda|+\aleph_0})^{+}$ with respect to models on ordinals equipped with the generalized Icard topologies $\mathcal{I}_{\lambda}$ and ${\tau_{c}}_{+\lambda}$. Specifically, we show that for such languages there exists a $\gl$-consistent set of formulas having neither $(\Theta, \mathcal{I}_{\lambda})$-model nor $(\Theta, {\tau_{c}}_{+\lambda})$-model.
We also introduce two kinds of natural classes of topological spaces, called \emph{ $\lambda$-bouquet spaces} and \emph{ultralinear $\lambda$-bouquet spaces}, and prove that they yield strong completeness of $\mathsf{GL}$ and $\mathsf{GL}.3$ respectively for languages of  cardinality $\lambda$.
\end{abstract}

\section{Introduction}

The G\"odel-L\"ob provability logic ($\mathsf{GL}$) is a propositional modal logic in which the modal operator $\Box$ is interpreted as a provability predicate within any formal theory $T$
capable of describing the arithmetic of natural numbers, such as Peano Arithmetic. A formula $\Box\varphi$ is read as ``$\varphi$ is provable in $T$," while $\lozenge\varphi$ is interpreted as ``$\varphi$ is consistent with $T$." Solovay showed in \cite{solovay}, that $\gl$ is complete for such interpretation. In the same paper he considered the interpretation of $\Box\varphi$ as ``$\varphi$ holds in every transitive model of $\zfc$'' and provided a completeness result for the extension $\gl.3$ of $\gl$ with respect to this interpretation.

The Kripke semantics of $\gl$ has been extensively studied. Segerberg \cite{segerberg} established that $\mathsf{GL}$ is sound and complete with respect to the class of all transitive and converse well-founded Kripke frames. In fact, it suffices to restrict attention to finite transitive irreflexive trees. Similarly, $\gl.3$ is complete with respect to finite strict linear orders. However, both $\mathsf{GL}$ and $\gl.3$ fail to be strongly complete with respect to their Kripke semantics.

Subsequently, Esakia \cite{esakia} observed that the modal operator $\lozenge$ behaves analogously to the derivative operator in topological scattered spaces. He then proved that $\mathsf{GL}$ is strongly complete with respect to the class of all scattered spaces.

In 1990, Blass \cite{blass}  extended Esakia's work by interpreting modal operators via filters associated with uncountable cardinals, adopting a set-theoretic perspective. He investigated the completeness of $\mathsf{GL}$ with respect to end-segment filters and closed unbounded (club) filters. Within ZFC, Blass proved completeness for end-segment filters, a  result that also establishes completeness of $\mathsf{GL}$ for all ordinals $\alpha\geq \omega^\omega$ under the interval topology (independently shown by Abashidze \cite{abashidze1985ordinal}). For club filters, Blass demonstrated completeness under G\"odel's axiom of constructibility, specifically Jensen's square principle $\Box_\kappa$ for all uncountable cardinals $\kappa<\aleph_\omega$. Furthermore, he showed that the failure of completeness for club filters is equiconsistent with the existence of a Mahlo cardinal. Following the approach of Blass, the first and the third author showed consistency of the completeness of $\gl$ for the normal topology \cite{GZ}, whereas Aguilera and the second author showed that it is consistent that $\gl.3$ is complete for the normal topology and that the Axiom of Determinacy implies completeness of $\gl.3$ for the club topology \cite{AgSt24}.

The Abashidze--Blass theorem initiated a new line of research into the
completeness of $\mathsf{GL}$ and also its polymodal extensions $\mathsf{GLP}$ with
respect to the natural topologies on ordinals (see, e.g., \cite{bagaria2019, beklemishev2013,aguilera2023topological}). The polymodal version of $\gl.3$, namely $\mathsf{GLP}.3$, was recently shown to be the logic of correct submodels in the setting of $V=L$ \cite{AgPa}.

By generalizing the notion of Icard topology introduced in \cite{icard09} for studying polymodal provability logic $\mathsf{GLP}$, \cite{aguilera2017strong} showed that by starting from a scattered space $\mathfrak{X}=(X,\tau)$ one can obtain a class of spaces  $\mathfrak{X}_{+\lambda}=(X,\tau_{+\lambda})$ such that $\mathsf{GL}$ is strongly complete with respect to $\mathfrak{X}_{+\lambda}$ for each $\lambda <\Lambda$ where $\Lambda$ depends on $\mathfrak{X}$.
More specifically, $\mathsf{GL}$ is strongly complete with respect to an ordinal $(\Theta,\mathcal{I}_\lambda)$ for $e^\lambda\omega<\Theta$ \cite[Corollary 6.14]{aguilera2017strong} and $(\Theta,{\tau_c}_{+\lambda})$ for $\aleph_{e^\lambda\omega+1}<\Theta$ \cite[Corollary 6.15]{aguilera2017strong}. As a special case, when $\lambda=1$, we have strengthening of Abashidze--Blass theorem.

However, the method used in \cite{aguilera2017strong} relies heavily on the fact that the set of propositional variables is countable, and they raised the question of whether for any set of propositional variables $\mathbb{P}$ with $|\mathbb{P}|\geq\aleph_1$, there is a natural topological space $\mathfrak{X}$ with respect to which $\mathsf{GL}(\mathbb{P})$ is strongly complete.

In this paper we study the strong completeness of $\mathsf{GL}$ in the setting of uncountable languages. We first show that strong completeness
 of $\mathsf{GL}(\mathbb{P})$ fails for generalized Icard topology on ordinals, by establishing the following theorem.

\begin{thm}\label{incom uncount1}
	Let $\lambda$ be any ordinal and $\kappa = |\lambda| +\aleph_0$.
	Let $\mathbb{P}=\{p_i \; : \; i<  (2^\kappa)^+   \}$ be a set of propositional variables with $|\mathbb{P}|= (2^\kappa)^+$.
	Let \[ \Gamma = \{ \Diamond p_i \; : \; i < (2^\kappa)^+ \} \cup \{ \Box(p_i \to \Diamond p_j) \;  : \; i < j < (2^\kappa)^+ \}. \]
	Then $\Gamma$  is not satisfiable on any $(\Theta,\mathcal{I}_\lambda)$ or $(\Theta, {\tau_c}_{+\lambda})$.
\end{thm}

We also introduce  natural classes of topological spaces, called $\lambda$-bouquet spaces and ultralinear $\lambda$-bouquet spaces, that yield strong completeness of $\mathsf{GL}(\mathbb{P})$ and $\mathsf{GL}.3(\mathbb{P})$ for uncountable languages $\mathbb{P}$ respectively. This construction generalizes the notion of an $\omega$-bouquet, originally defined in \cite{aguilera2017strong} to establish the strong completeness of $\mathsf{GL}$.\begin{thm}\label{mainthm}
  Suppose $\lambda\geq \aleph_0$ and $|\mathbb{P}|=\lambda$. Then:
\begin{enumerate}
\item	The logic $\mathsf{GL}(\mathbb{P})$  is strongly complete with respect to $\lambda$-bouquet spaces.
\item	The logic $\mathsf{GL}.3(\mathbb{P})$  is strongly complete with respect to ultralinear $\lambda$-bouquet spaces.
\end{enumerate}
\end{thm}
	The paper is organized as follows. In Section \ref{Preliminaries}, we give some preliminaries from set theory and provability logic.
  In Section \ref{Negative Results}, we apply the Erdo\"s-Rado partition theorem to prove Theorem \ref{incom uncount1}. In Section \ref{sec:bouquet},
we introduce the concept of $\lambda$-bouquet spaces and ultralinear $\lambda$-bouquet spaces and prove Theorem \ref{mainthm}. We conclude the paper with some questions.
	
\section{Preliminaries}
\label{Preliminaries}

\subsection{Set Theory}
Assuming familiarity with elementary ordinal arithmetic, we recall several  foundational set-theoretic definitions and notational conventions required for the subsequent development. For more detail on set theory see for example \cite{jech2003set}.

Let $\mathrm{Ord}$ denote the class of all ordinal numbers, with elements denoted by $\alpha,\beta,\dots$.
For any two ordinals $\alpha$ and $\beta$, if $\alpha<\beta$, then there is a unique ordinal $\gamma$ such that $\alpha+\gamma=\beta$. We denote this unique $\gamma$ by $-\alpha+\beta$.

\begin{dfn}
	Let $\alpha$ and $\beta$ be two limit ordinals.
	An increasing $\beta$-sequence $\langle\alpha_\xi \; : \; \xi<\beta\rangle$ is \textit{cofinal} in $\alpha$ if $\lim_{\xi\to\beta}\alpha_\xi=\sup_{\xi<\beta} \alpha_\xi =\alpha$.
	Similarly, a subset $A\subset \alpha$ is cofinal in $\alpha$ if $\sup A=\alpha$.
	
	For an infinite limit ordinal $\alpha$, the \emph{cofinality} of $\alpha$, denoted by $\mathrm{cf}(\alpha)$, is the least limit ordinal $\beta$ for which there exists an increasing $\beta$-sequence $\langle\alpha_\xi \; :\;  \xi<\beta\rangle$ cofinal in $\alpha$.
\end{dfn}

For a limit ordinal $\alpha$, a subset $A\subset \alpha$ is called \textit{bounded} if $\sup A<\alpha$, and \textit{unbounded} if $\sup A=\alpha$.

\begin{lem}\label{cf sum}
	Let $\alpha$ and $\beta$ be two ordinals. Then $\mathrm{cf}(\alpha+\beta) =\mathrm{cf}(\beta)$.
\end{lem}

Two classes of functions on ordinals, known as \textit{hyperexponentials} and \textit{hyperlogarithms}, play a central role in the study of provability logic. They were first introduced in \cite{fernandez2013models,fernandez2014well}; see \cite{fernandez2013hyperations}  for a detailed account of their properties.

\begin{dfn}[Hyperexponential functions]
	Let $e:\mathrm{Ord}\to \mathrm{Ord}$ be a normal function defined by $\xi\mapsto -1+\omega^\xi$, called the \textit{exponential function}.
	The unique family of normal functions $(e^\zeta)_{\zeta\in\mathrm{Ord}}$ is called \textit{hyperexponentials} if they satisfy:
	\begin{enumerate}[i.]
		\item $e^1=e$,
		\item $e^{\alpha+\beta}=e^\alpha\circ e^\beta$,
		\item if $(f^\xi)_{\xi\in\mathrm{Ord}}$ is any family of functions satisfying i and ii, then $e^\alpha\beta\leq f^\alpha\beta$, for all $\alpha,\beta\in\mathrm{Ord}$.
	\end{enumerate}
\end{dfn}

\begin{lem}\cite[Corollary 2.7]{aguilera2017strong}\label{cf l lem}
	For any ordinal $\xi$ we have
	\[\mathrm{cf}(e^\lambda \xi) =
	\begin{cases}
		\mathrm{cf}(\xi) &  \xi \text{ is a limit ordinal} \\
		\max\{\mathrm{cf}(\lambda), \aleph_0\} & \text{o.w.}
	\end{cases}\]
\end{lem}

By Cantor normal form theorem for ordinals, it is easy to see that, for any non-zero ordinal $\xi$, there exist $\alpha$ and a unique $\beta$ such that $\xi=\alpha+\omega^\beta$. This unique $\beta$ is called the \emph{end logarithm} of $\xi$ and is denoted by $\ell(\xi)$.

\begin{dfn}[Hyperlogarithms]
	The hyperlogarithms $(\ell^\xi)_{\xi\in\mathrm{Ord}}$ are the unique family of initial functions that satisfy:
		\begin{enumerate}[i.]
			\item $\ell^1=\ell$,
			\item $\ell^{\alpha+\beta}= \ell^\beta\circ \ell^\alpha$,
			\item if $(f^\xi)_{\xi\in\mathrm{Ord}}$ is any family of functions satisfying i and ii, then $\ell^\alpha\beta\geq f^\alpha\beta$, for all $\alpha,\beta\in\mathrm{Ord}$.
		\end{enumerate}
\end{dfn}

\begin{lem}\cite[Lemma 2.11]{aguilera2017strong} \label{l lem}
	For any $\xi\in\mathrm{Ord}$ we have
	\begin{enumerate}[i.]
		\item $\ell\xi \leq\xi$.
		\item If $\xi$ and $\delta$ are non-zero, then $\ell^\xi(\gamma+\delta)=\ell^\xi\delta$.
		\item\label{l lem3} If $\lambda = \eta + \omega^\beta$ is limit and $\xi$ is an ordinal, then there exists $\sigma<\lambda$ such that for all large $\zeta\in [\sigma,\lambda)$ we have
		$ \ell^\zeta \xi = e^{\omega^\beta} \ell^\lambda \xi $.
	\end{enumerate}
\end{lem}
One can  prove that $\ell^\xi$ is a left inverse of $e^\xi$ for any $\xi\in\mathrm{Ord}$ (see \cite[Theorem 9.1]{fernandez2013hyperations}).
\begin{lem}{\cite[Lemma 2.12]{{aguilera2017strong}}} \label{left-ord}
	If $\xi<\zeta$, then $\ell^\xi e^\zeta = e^{-\xi+\zeta}$ and $\ell^\zeta e^\xi=\ell^{-\xi+\zeta}$. Furthermore, if $\alpha< e^\xi\beta$, then $\ell^\xi\alpha<\beta$, and if $\alpha<\ell^\xi\beta$, then $e^\xi\alpha<\beta$.
\end{lem}


The Erd\H{o}s--Rado theorem is a central result in infinitary combinatorics and set theory, extending Ramsey's  theorem to infinite sets and higher cardinalities. Before stating the theorem, we recall the relevant notation.

The \textit{Beth function} is defined inductively as follows: 
\[
\beth_0(\kappa) = \kappa, \qquad
\beth_{\alpha+1}(\kappa) = 2^{\beth_\alpha(\kappa)}, \qquad
\beth_\alpha(\kappa) = \sup_{\beta<\alpha} \beth_\beta(\kappa) \quad (\alpha \text{ a limit ordinal}).
\]
For cardinals $\lambda,\kappa,\delta$ and a positive integer $n$, the partition relation
$$\lambda \to (\kappa)_\delta^{n}$$
asserts that for every colouring of the  $n$-element subsets of a set of cardinality $\lambda$ with $\delta$ colours, there is a subset of size $\kappa$ whose  $n$-element subsets are monochromatic.

\begin{thm}[Erd\H{o}s--Rado]\label{erdos-rado}
	Let $\kappa$ be an infinite cardinal and
	$n$ be a positive integer. Then
	$\beth_n(\kappa)^+\to (\kappa^+)_{\kappa}^{n+1}.$
	In particular,
	 $(2^{\kappa})^+ \to (\kappa^+)^2_{\kappa}$.
\end{thm}

In the special case $n=1$ and $\kappa=\aleph_0$, the theorem states: if the pairs of a set of cardinality just above the continuum are coloured with countably many colours, then there exists an uncountable subset all of whose pairs receive the same colour. Equivalently, if $f$ is a colouring of the $2$-element subsets of a set of size $\mathfrak{c}^+=(2^{\aleph_0})^+$ into $\aleph_0$ colours, then there exists a subset of cardinality $\aleph_1$ that is homogeneous for $f$.

 Filters and ultrafilters are important combinatorial objects, who appear in different areas of mathematics, in this paper we need only basic facts and definitions:

\begin{dfn}\label{def:filter}
  Let $X$ be a set. We say that $F\subset \mathcal{P}(X)$ is a \emph{filter (on $X$)} if the following hold:
  \begin{enumerate}
    \item $X\in F$ and $\emptyset \notin F$;
    \item $X,Y\in F$ implies $X\cap Y\in F$.
  \end{enumerate}
  We say that a filter $F$ is an \emph{ultrafilter} if for any $A\subset X$, either $A\in F$ or $X\setminus A\in F$. A filter $F$ on $X$ is called \emph{principal} if it is generated by a single subset i.e. $F=\set{A\subset X: A\supset X_0}$ for $X_0\subset X$ and is called \emph{non-principal} otherwise.

  We denote by $F^+$ the subset of $\mathcal{P}(X)$ such that $A\in F^+$ if and only if $A\cap B\ne \emptyset$ whenever $B\in F$, we call $F^+$ \emph{the set of $F$-positive subsets}. Note that $F\subset F^+$, moreover $F=F^+$ if and only if $F$ is an ultrafilter.

  Given a filter $F$ on $X$, we say that a statement $\varphi$ holds \emph{$F$-almost everywhere} (or \emph{$F$-a.e.}) if $\set{x\in X: \varphi(x)}\in F$.
\end{dfn}

A natural example of a filter is the \emph{co-bounded filter} on an infinite cardinal $\kappa$, which is the set of all subsets of $\kappa$, whose complement is bounded in $\kappa$. The set of unbounded subsets of $\kappa$ is the set of positive subsets with respect to the co-bounded filter. Existence of non-principal ultrafilters guaranteed by the Axiom of Choice, however, it is consistent with $\zf$ that there are no non-principal ultrafilters on $\omega$, e.g. under the Axiom of Determinacy.

\subsection{Provability Logic $\mathsf{GL}$ and Its Semantics}

Let $\kappa$ be an infinite cardinal, and define $\mathbb{P}=\{p_\xi :  \xi<\kappa\}$ as a set of propositional variables indexed by ordinals less than $\kappa$.
The corresponding modal language, denoted by $\mathcal{L}(\mathbb{P})$, is constructed from  $\mathbb{P}$ using the standard Boolean connectives $\neg,\wedge$ together with the modal operator $\Box$. The remaining Boolean connectives $\vee, \to$ are defined in the usual way, and the operator $\lozenge$ is introduced as the dual of $\Box$. We also use $\mathcal{L}(\mathbb{P})$ to refer to the set of all well-formed formulas generated in this language.

The provability logic $\mathsf{GL}(\mathbb{P})$ is defined within $\mathcal{L}(\mathbb{P})$ and is axiomatized by the following schemata and inference rules:
\begin{itemize}
	\item All propositional tautologies,
	\item $\textbf{K:} \Box(\varphi\to \psi) \to (\Box\varphi\to \Box\psi)$,
	\item $\textbf{L\"ob:} \Box(\Box\varphi\to \varphi)\to \Box\varphi$,
	\item Modus ponens,
	\item Necessitation: $\frac{\varphi}{\Box\varphi}$. 
\end{itemize}
When $\mathbb{P}$ is countable, we denote the corresponding logic simply by $\mathsf{GL}$.

One of the most widely studied semantics for modal logic, and in particular for provability logic, is relational semantics, commonly referred to as Kripke semantics.
A \textit{Kripke frame} is a relational structure $\mathfrak{F}=(W,R)$, where $W$ is a non-empty set and $R$ is a binary relation on $W$.
A \textit{tree} is a frame $(T,R)$ where $R$ is transitive and, for each $t\in T$, the set of its $R$-predecessors $R^{-1}(t)=\{w\in T \; : \;  wRt\}$ is well-ordered under $R$.
Moreover, $(T,R)$ contains a unique minimal element with respect to $R$, called the \textit{root}, denoted by $\mathrm{root}(T)$.
Also, for each $t\in T$, the set of all its immediate successors, i.e., those $s\in T$ such that $tRs$ and there is no $u\in T$ with $tRu$ and $uRs$, is denoted by $\mathrm{Suc}(t)$.

A \textit{Kripke model} is obtained by equipping a frame with a valuation function $\lbr\cdot\rbr:\mathbb{P}\to \mathcal{P}(W)$.
This valuation is extended to all formulas in $\mathcal{L}(\mathbb{P})$ in such a way that it respects the Boolean connectives, and satisfies
$\lbr\lozenge\varphi\rbr = R^{-1}(\lbr\varphi\rbr)$.

A relation $R$ on a set $W$ is called \emph{converse well-founded} if every non-empty subset of $W$ has an $R$-maximal element.
It is well known that $\mathsf{GL}$ is complete with respect to the class of transitive and converse well-founded Kripke frames. More specifically:

\begin{thm}\cite{segerberg}
	$\mathsf{GL}$ is sound and complete with respect to the class of all
	finite irreflexive trees.
\end{thm}

Nevertheless, $\mathsf{GL}$ fails to be strongly complete with respect to its relational semantics. For instance, consider
\begin{equation}\label{eq:counterexample-intro}\tag{$*$}
\Gamma_*=\{\lozenge p_0\} \cup\{\Box (p_i\to \lozenge p_{i+1}) \; : \;   i\in\omega\}.
\end{equation}
This set is finitely satisfiable, yet it is not satisfiable in any converse well-founded tree.

While Kripke semantics has been the traditional framework for modal and provability logics, recent attention has shifted toward topological semantics, where $\mathsf{GL}$ exhibits particularly favourable properties. In contrast to the Kripke setting, $\mathsf{GL}$ is strongly complete with respect to topological semantics.

Let $\mathfrak{X}=(X,\tau)$ be a topological space.
A \textit{limit point} of a set $A\subseteq X$ is a point $x\in X$ such that for every open neighbourhood $U$ of $x$ the intersection $A\cap (U-\{x\})$ is non-empty. The \emph{derived set} of $A$, denoted by $d(A)$, is the set of all limit points of $A$.
A point $x\in X$ is \textit{isolated} in $A$ if it is not a limit point of $A$.

The transfinite iteration $(d^\xi)_{\xi\in \mathrm{Ord}}$ of derived-set operator $d$ is defined by
\[
d^0 A = A, \qquad
d^{\xi+1} A = d(d^\xi A), \qquad
d^\alpha A = \bigcap_{\xi<\alpha} d^\xi A \quad \text{for limit ordinals } \alpha.
\]
For any topological space $\mathfrak{X}=(X,\tau)$, the sequence $\langle d^\xi(X) \; : \; {\xi\in \mathrm{Ord}}\rangle$ is a decreasing sequence, so there exists a least ordinal $\alpha$ such that $d^\alpha(X) = d^{\alpha+1}(X)$.
For $x\in X$, the \textit{rank} of $x$, denoted by $\rho_\mathfrak{X}x$ (or $\rho_\tau x$), is the least ordinal $\xi$ such that $x\not\in d^{\xi+1} X$, if such an ordinal exists.

\begin{dfn}
	A topological space $\mathfrak{X}$ is \textit{scattered} if each non-empty subset of $X$ has an isolated point.
\end{dfn}
It is shown by Cantor that a topological space $\mathfrak{X}$ is scattered if and only if there exists an ordinal $\xi$ such that $d^\xi X=\emptyset$. In this case, $\rho_\mathfrak{X}x$ exists for all $x\in X$.
The rank of $\mathfrak{X}$ is then defined by $\rho\mathfrak{X} =\sup_{x\in X}(\rho_\tau x+1)$.

\begin{exm}\cite{aguilera2017strong}\label{exam ord}
	\begin{itemize}
		\item The \emph{upset topology} on a relational structure $(T,R)$ contains all $U\subseteq T$ such that $U\in \tau_R$ provided that $R(x)\subseteq U$ for each $x\in U$.
		Any leaf of an irreflexive tree $T$ is an isolated point.
		Hence, if $T$ is a converse well-founded tree, then $(T,\tau_R)$ is scattered.
		\item The \emph{left topology} over an ordinal $\Theta$, denoted by $\mathcal{I}_0$, is generated by all sets of the form $[0,\beta]$. Then $(\Theta,\mathcal{I}_0)$ is scattered, with the least element of each set as its isolated point. Then for each $\alpha\in\Theta$ we have $\rho_{\mathcal{I}_0}(\alpha)= \alpha$.
		\item The \emph{interval topology} (or \emph{order topology}), on an ordinal $\Theta$, denoted by $\mathcal{I}_1$, is generated by all sets of the form $(\alpha,\beta]$ and $[0,\beta]$.
		Then $(\Theta,\mathcal{I}_1)$ is a scattered space where all successor ordinals are isolated points.
		In this case, $\rho_{\mathcal{I}_1}(\alpha)=\ell(\alpha)$.
		\item The \emph{club topology} on an ordinal $\Theta$, denoted by $\tau_c$, is the unique topology such that if $\mathrm{cf}(\alpha)\leq\omega$, then $\alpha$ is an isolated point and if $\mathrm{cf}(\alpha)>\omega$, then $U\subseteq\Theta$ contains a neighbourhood of $\alpha$ iff $\alpha\in U$ and $U$ contains a club (i.e. closed and unbounded in $\mathcal{I}_1$) in $\alpha$. Then $(\Theta,\tau_c)$ is a scattered space and for non-zero $\alpha$, $\rho_{\tau_c}(\xi)=\alpha$ iff $\mathrm{cf}(\xi)=\Omega_\alpha$, where $(\Omega_\alpha)_{\alpha\in\mathrm{Ord}}$ is an enumeration of all infinite regular cardinals.
	\end{itemize}
\end{exm}

\begin{dfn}[Topological Semantics]
	A topological model based on a topological space $\mathfrak{X}=(X,\tau)$ is a triple $\mathfrak{M}=(X,\tau,\lbr\cdot\rbr)$ where $\lbr\cdot\rbr:\mathbb{P}\to \mathcal{P}(X)$ is a valuation function.
	The valuation extends naturally to Boolean connectives, and	$\lbr\lozenge\varphi\rbr= d(\lbr\varphi\rbr)$.
\end{dfn}
We write $\mathfrak{M},x\models\varphi$, when $x\in \lbr\varphi\rbr$.
A formula $\varphi$ is \textit{satisfiable} in $\mathfrak{X}$ if there is a model $\mathfrak{M}$ based on $\mathfrak{X}$ and a point $x\in X$ such that $\mathfrak{M},x\models\varphi$.
Also, $\varphi$ is \textit{valid} in $\mathfrak{X}$, denoted by $\mathfrak{X}\models\varphi$, if $\lbr\varphi\rbr=X$ for any valuation $\lbr\cdot\rbr$ on $\mathfrak{X}$.

\begin{thm}\cite{esakia}
	$\mathsf{GL}$ is sound and (strongly) complete with respect to the class of all scattered spaces.
\end{thm}
Esakia further showed that there exists a countable scattered space $\mathfrak{X}$ such that $Log(\mathfrak{X})=\mathsf{GL}$.
Moreover, \cite{aguilera2017strong} strengthened results of Esakia and Abashidze-–Blass by proving:
\begin{thm}\cite{aguilera2017strong}\label{ag strong}
	$\mathsf{GL}$ is strongly complete with respect to any ordinal space $(\Theta,\mathcal{I}_1)$ for any $\omega^\omega<\Theta$.
\end{thm}
This theorem holds for countable languages. In the remainder of this paper we investigate the strong completeness of $\mathsf{GL}(\mathbb{P})$ for uncountable sets of propositional variables $\mathbb{P}$ with respect to ordinal spaces.
To this end, we recall some basic definitions and facts.

Let $\mathfrak{X}=(X,\tau)$ be a topological space and $X'\subseteq X$. The structure $\mathfrak{X}'=(X',\tau')$ is called a \textit{subspace} of $\mathfrak{X}$ if $\tau'=\{O\cap X' \;  : \;  O\in\tau\}$.
If $X'$ is an open subset of $X$, then $\mathfrak{X}'$ is an \textit{open subspace} of $\mathfrak{X}$.
Similarly, given a model $\mathfrak{M}=(X,\tau,\lbr\cdot\rbr)$ and a subset $X'\subseteq X$, the structure $\mathfrak{M}'=(X',\tau,\lbr\cdot\rbr')$ is a \textit{submodel} of $\mathfrak{M}$ if $\lbr p\rbr'=\lbr p\rbr\cap X'$, for each $p\in\mathbb{P}$.
If $\mathfrak{M}'$ is an open submodel of $\mathfrak{M}$, then by induction on the complexity of formulas, one can verify that for any $\varphi$ and any $x\in X'$ we have $\mathfrak{M},x\models\varphi$ iff $\mathfrak{M}',x\models\varphi$.
Consequently, if $\mathfrak{X}\models\varphi$, then $\mathfrak{X}'\models\varphi$.
In particular, for an ordinal space $(\Theta,\mathcal{I}_1)$ if $\alpha<\Theta$, then the interval $[0,\alpha]$ forms an open subspace $(\alpha,\mathcal{I}_1)$.

For an ordinal $\Theta$ and $\alpha,\beta<\Theta$, and a model $\mathfrak{M}$ based on $\Theta$, we write $\mathfrak{M},[\alpha,\beta]\models \varphi$, to mean that $\mathfrak{M},\gamma\models\varphi$ for any $\gamma\in[\alpha,\beta]$. Similar notation will be used for open and half-open intervals.

\subsection{Generalized Icard Topology}\label{icard sec}
Theorem~\ref{ag strong} is a specific version of Corollary 6.14 in \cite{aguilera2017strong}, where the strong completeness of $\mathsf{GL}$ is proved for  any ordinal $\Theta> e^\lambda\omega$ equipped with some general topology named generalized Icard topology.
In this subsection we review the definition of generalized Icard topologies over any scattered space, and some of its properties.

To ease notation, in some cases, the interval $[0,\alpha]$ will be denoted by $(-1,\alpha]$.

\begin{dfn}
	Let $\mathfrak{X}=(X,\tau)$ be a scattered topological space with $\rho_\tau\mathfrak{X}=\Theta$.
	For any $\lambda\in\mathrm{Ord}$, the \textit{generalized Icard topology} on $X$, denoted by $\tau_{+\lambda}$, is the least topology on $X$ containing $\tau$ and all sets of the form
	$$(\alpha,\beta]_\xi^\mathfrak{X} =\{ x\in X \; : \;  \alpha<\ell^\xi\rho_\tau x\leq\beta\},$$
	for some $-1\leq \alpha<\beta\leq \Theta$ and some $\xi<\lambda$.
	This space is denoted by $\mathfrak{X}_{+\lambda}=(X,\tau_{+\lambda})$.
\end{dfn}

Specifically, for ordinal space $(\Theta,\mathcal{I}_0)$ and $\lambda\in\mathrm{Ord}$, the generalized Icard topology  $(\mathcal{I}_0)_{+\lambda}$ is the least topology on $\Theta$  generated by all sets of the form
$$(\alpha,\beta]_\xi =\{ \theta \; : \;  \alpha<\ell^\xi\theta\leq\beta\},$$
for some $-1\leq \alpha<\beta\leq \Theta$ and some $\xi<\lambda$. We denote this space simply by $\Theta_\lambda =(\Theta,\mathcal{I}_\lambda)$ and the rank of it by $\rho_\lambda$.
Then, it is easy to see that $(\mathcal{I}_0)_{+1}= \mathcal{I}_1$ is the interval topology.

\begin{lem}\cite[Lemma~5.5]{aguilera2017strong} 
	Let $\mathfrak{X}$ be a scattered space and $\lambda\in\mathrm{Ord}$. Then $\rho_\lambda = \ell^\xi\rho_\mathfrak{X}$.
	Specifically, for any ordinals $\Theta$ and any $\theta<\Theta$, we have $\rho_\lambda\theta=\ell^\lambda \theta$.
\end{lem}


It would be beneficial to know how one can consider a basic open neighbourhood of any ordinal $\alpha$ in Icard topologies.

For any finite subset $A\subseteq\mathrm{Ord}$, a function $r:A\to \mathrm{Ord}$ is called a \emph{simple function}.
Let $\alpha$ be any ordinal and $r$ be a simple function such that $r(\xi)< \ell^\xi\alpha$ for all $\xi\in \mathrm{dom}(r)$. Then for any $x\in X$ of rank $\alpha$ put $B_r^\mathfrak{X}(x) = \bigcap_{\xi\in \mathrm{dom}(r)} (r(\xi),\ell^\xi\alpha]_\xi^\mathfrak{X}.$

Note that if $\lambda> \max(\mathrm{dom}(r))$, then $B_r^\mathfrak{X}(x)$ is $\tau_{+\lambda}$-open.

\begin{lem}\cite[Lemma~5.11]{aguilera2017strong}
	Let $\mathfrak{X}=(X,\tau)$ be a scattered space and $0<\lambda$ an ordinal.
	Then for any $x\in X$ with $\rho_\tau(x)=\xi$ and $0<\ell^\lambda\xi$, the sets of the form $U\cap B_r^\mathfrak{X}(x)$ with $x\in U\in\tau$ and $\mathrm{dom}(r)\subseteq \lambda$ form a neighbourhood base for $x$ in $\tau_{+\lambda}$.
	Furthermore, for $\mathfrak{X}=(\Theta,\mathcal{I}_0)$, we can take $U=\Theta$.
\end{lem}

\subsection{The Logic $\gl.3$}

\begin{dfn}
  The logic $\gl.3(\mathbb P)$ is the minimal extension of $\gl(\mathbb P)$ that contains the axiom $.3$, where $.3$ is:
  \[
    \Box(\Box\varphi \to \psi)\lor\Box(\Box\psi\land\psi \to \varphi).
  \]
  When $\mathbb{P}$ is countable, we write $\gl.3$ instead of $\gl.3(\mathbb P)$.
\end{dfn}

\begin{dfn}
  If $X$ is a set, we say that $R\subset X\times X$ is a \emph{strict linear order}, if $R$ is transitive, antisymmetric, irreflexive and total, where by the latter we mean that for all $x,y\in X$ either $xRy$ or $yRx$.
\end{dfn}

Note that if $X$ is finite, then $(X,R)$ is order-isomorphic to $(n,<)$, where $n=\set{0,\dots,n-1}$ is the size of $X$, ordered by $<$ in the usual sense.

\begin{fact}[See e.g.~\cite{Boolos_1994}]
  $\gl.3$ is sound and complete with respect to the class of finite strict linear orders.
\end{fact}

As with $\gl$, $\gl.3$ fails to be strongly complete with respect to its Kripke semantics, i.e., the class of finite strict linear orders. This is witnessed by the set $\Gamma_*$ from \eqref{eq:counterexample-intro}.

\section{Some Negative Results}\label{Negative Results}

As mentioned above, $\mathsf{GL}$ is not strongly complete with respect to Kripke semantics, but it is strongly complete with respect to scattered spaces. Specifically it is strongly complete with respect to $(\Theta, \mathcal{I}_\lambda)$ for $\lambda>0$ and $\Theta>e^\lambda\omega$,
and $(\Theta, {\tau_c}_{+\lambda})$ for $\lambda>0$ and $\Theta>\aleph_{e^\lambda\omega+1}$.
However, there are some negative results regarding strong completeness of $\mathsf{GL}$ with respect to some specific topological spaces.

\begin{prop}\label{g-delta}
	$\mathsf{GL}$ is not strongly complete with respect to any scattered space $\mathfrak{X}$ in which every $G_\delta$ set is open.
\end{prop}
For the proof, see \cite[Proposition~2.6]{aguilera2023topological}, where it is also shown that:
\begin{cor}
	$\mathsf{GL}$ is not strongly complete with respect to topologies on ordinals induced by countably complete filters, such as the club topology.
\end{cor}

In the following subsections we investigate whether $\mathsf{GL}(\mathbb{P})$ is strongly complete with resect to an ordinal space, for uncountable $\mathbb{P}$.

\subsection{Incompleteness with $\mathcal{I}_1$}
We begin by showing that, for a certain set $\mathbb{P}$ of uncountable size, the logic $\mathsf{GL}(\mathbb{P})$ fails to be strongly complete with respect to ordinals equipped with the interval topology.
We start with the following easy fact.
\begin{fact}\label{fact}
	Let $\Theta$ be an ordinal and $\mathfrak{M}$ a model based on $(\Theta,\mathcal{I}_1)$. For any $\alpha\in\Theta$ and formula $\varphi$:
	\begin{itemize}
		\item $\mathfrak{M},\alpha\models\lozenge\varphi$ iff $\{\beta<\alpha \; :\; \mathfrak{M},\beta\models\varphi\}$ is cofinal in $\alpha$.
		\item If $\langle\alpha_i \; : \; i<\mathrm{cf}(\alpha)\rangle$ is a cofinal sequence of $\alpha$, then  $\mathfrak{M},\alpha\models\Box\varphi$ iff $\mathfrak{M},(\alpha_i,\alpha)\models\varphi$ for some $i<\mathrm{cf}(\alpha)$.
	\end{itemize}
\end{fact}
\begin{proof}
	Let $B=\{\beta<\alpha\; : \; \mathfrak{M},\beta\models\varphi\}$.
	If $B$ is cofinal in $\alpha$, then $\sup B=\alpha$. Hence, for any basic $\mathcal{I}_1$-open neighbourhood $(\gamma,\alpha]$ of $\alpha$, there is a $\beta\in (\gamma,\alpha)$ with $\beta\in B$. Thus, $\mathfrak{M},\alpha\models \lozenge\varphi$.
	If $B$ is not cofinal in $\alpha$, then $\sup B=\beta'<\alpha$. The neighbourhood $(\beta',\alpha]$ of $\alpha$ contains no points satisfying $\varphi$, so $\mathfrak{M},\alpha\not\models\lozenge\varphi$.
	
	The second statement follows immediately by duality of $\Box$ and $\lozenge$.
\end{proof}
We now construct a set of formulas in an uncountable language that admits no ordinal model.
For any set $A$ and cardinal $\lambda$, let $[A]^\lambda$ denote the set of all subsets of $A$ of cardinality $\lambda$. Also, let $\mathfrak{c}=2^{\aleph_0}$.

\begin{prop}\label{incom countable}
	Let $\mathbb{P}=\{p_i \; : \;  i<\mathfrak{c}^+\}$ be a set of propositional variables.
	Then the set
	\[
	\Gamma = \{\lozenge p_i \; : \;  i<\mathfrak{c}^+\} \cup \{\Box(p_i \to \lozenge p_j) \; : \;  i<j<\mathfrak{c}^+\}
	\]	
	has no model based on any ordinal with interval topology.
\end{prop}

\begin{proof}
	First observe that $\Gamma$ is consistent: every finite subset of $\Gamma$ is easily seen to be satisfiable.
	
	Assume, for contradiction, that $\mathfrak{M},\alpha\models\Gamma$ for some model $\mathfrak{M}$ based on an ordinal $(\Theta,\mathcal{I}_1)$ with $\alpha<\Theta$.
	Clearly, $\alpha$ must be a limit ordinal. We distinguish two cases according to the cofinality of $\alpha$.
	
	\begin{description}
		\item[Case 1] Suppose $\mathrm{cf}(\alpha)=\omega$.
		Let $\langle \alpha_n \; : \;  n<\omega\rangle$ be an increasing cofinal sequence in $\alpha$.
		Define $\mathbf{c}:[\mathfrak{c}^+]^2 \to \omega$ by,
		$$\mathbf{c}(i,j)=\min\{n \; : \;  \mathfrak{M},(\alpha_n,\alpha)\models p_i \to \lozenge p_j\},$$

		for $i<j$.
		This is well-defined since $\mathfrak{M},\alpha\models \Box(p_i \to \lozenge p_j)$ (see Fact~\ref{fact}).
		
		By the Erd\H{o}s--Rado theorem (Theorem~\ref{erdos-rado}), there exists $S\subseteq \mathfrak{c}^+$ of size $\aleph_1$ and $n_*<\omega$ such that $\mathbf{c}(i,j)=n_*$ for all $i<j$ in $S$.
		Hence, for all $i<j$ in $S$ and all $\beta\in (\alpha_{n_*},\alpha)$, we have $\mathfrak{M},\beta\models p_i \to \lozenge p_j$.
		
		Choose $i_0<i_1<i_2<\dots$ from $S$.
		Since $\mathfrak{M},\alpha\models \lozenge p_{i_0}$, there exists $\beta_0\in (\alpha_{n_*},\alpha)$ with $\mathfrak{M},\beta_0\models p_{i_0}$.
		Then $\mathfrak{M},\beta_0\models p_{i_0}\wedge \lozenge p_{i_1}$, so we can find $\beta_1\in (\alpha_{n_*},\beta_0)$ with $\mathfrak{M},\beta_1\models p_{i_1}$.
		Repeating this argument, we obtain a strictly decreasing sequence $\langle \beta_m \; :\;  m<\omega\rangle$ of ordinals, a contradiction.
		
		\item[Case 2] Suppose $\mathrm{cf}(\alpha)>\omega$.
		For each $n<\omega$, let $\gamma_n<\alpha$ be such that $\mathfrak{M},(\gamma_n,\alpha)\models p_n \to \lozenge p_{n+1}$.
		Set $\gamma=\sup_{n<\omega}\gamma_n$. Since $\mathrm{cf}(\alpha)>\omega$, we have $\gamma<\alpha$.
		Thus, for all $\beta\in (\gamma,\alpha)$ and all $n<\omega$, $\mathfrak{M},\beta\models p_n \to \lozenge p_{n+1}$.
		
		Proceeding as in Case~1, choose $\beta_0\in (\gamma,\alpha)$ with $\mathfrak{M},\beta_0\models p_0$.
		Then $\mathfrak{M},\beta_0\models p_0\wedge \lozenge p_1$, so there exists $\beta_1\in (\gamma,\beta_0)$ with $\mathfrak{M},\beta_1\models p_1$.
		Continuing inductively, we obtain a strictly decreasing sequence $\langle \beta_n \; : \; n<\omega\rangle$, again a contradiction.
	\end{description}
	So, it is not possible that $\mathfrak{M},\alpha\models\Gamma$.
\end{proof}

\subsection{Incompleteness with $\mathcal{I}_\lambda$}\label{icard incom}
In this part we show that for any ordinal $\lambda$ there is an uncountable language $\mathcal{L}(\mathbb{P})$ of high cardinality such that $\mathsf{GL}(\mathbb{P})$ is not strongly complete with respect to any ordinals with $\mathcal{I}_\lambda$.

As we saw, the proof of Proposition~\ref{incom countable} is based on cofinality of $\alpha$. To show the strong incompleteness for $\mathcal{I}_\lambda$ we need the same idea. So we have to know the relation between the cofinality of $\alpha$ and $\ell^\xi\alpha$.

\begin{lem}\label{cf >}
	Suppose $\mathrm{cf}(\alpha) > |\lambda| + \aleph_0$. Then for all $\xi < \lambda$
	\[ \mathrm{cf}(\ell^\xi \alpha) = \mathrm{cf}(\alpha). \]
\end{lem}
\begin{proof}
	The proof proceeds by induction on $\xi$, and for all ordinals $\alpha$ with $\mathrm{cf}(\alpha) > |\lambda| + \aleph_0$.
	
	First assume that $\xi = 1$ and $\alpha = \beta + \omega^\theta$ with $\mathrm{cf}(\alpha)>|\lambda| + \aleph_0$. Then, Lemma~\ref{cf sum} implies that $\theta > 0$, and hence by Lemma~\ref{cf l lem} and induction hypothesis we have
	\[ \mathrm{cf}(\alpha) = \mathrm{cf}(\omega^\theta) = \mathrm{cf}(\theta) = \mathrm{cf}(\ell \alpha) . \]
	
	Suppose that $\xi = \sigma + 1$ is a successor ordinal and $\mathrm{cf}(\ell^\sigma \alpha) = \mathrm{cf}(\alpha)$.
	So $\mathrm{cf}(\ell^\sigma \alpha) > |\lambda| + \aleph_0$, hence by induction hypothesis we have
	\[ \mathrm{cf}(\ell^{\sigma+1} \alpha) = \mathrm{cf}(\ell(\ell^\sigma \alpha)) = \mathrm{cf}(\ell^\sigma \alpha) = \mathrm{cf}(\alpha). \]
	
	Now let $\xi$ be a limit ordinal and $\mathrm{cf}(\ell^\sigma \alpha) = \mathrm{cf}(\alpha)$ for all $\sigma < \xi$ . Let $\xi = \eta + \omega^\beta$. Then by Lemma~\ref{l lem}~\ref{l lem3}, for large enough $\sigma < \xi$ we have
	\[ \ell^\sigma \alpha = e^{\omega^\beta} \ell^\xi \alpha, \]
	and then
	\[ \mathrm{cf}(\alpha) = \mathrm{cf}(\ell^\sigma \alpha) = \mathrm{cf}(e^{\omega^\beta} \ell^\xi \alpha). \]
	
	If $\ell^\xi \alpha$ is a limit ordinal, by Lemma~\ref{cf l lem},
	\[ \mathrm{cf}(e^{\omega^\beta} \ell^\xi \alpha) = \mathrm{cf}(\ell^\xi \alpha) \]
	and we are done. Otherwise
\[
\mathrm{cf}(e^{\omega^\beta} \ell^\xi \alpha) \leq \max\{\aleph_0, \mathrm{cf}(\omega^\beta)\} \leq |\lambda| + \aleph_0.
\]
	So $\mathrm{cf}(\alpha) \leq |\lambda| + \aleph_0$, a contradiction. Thus the Lemma follows.
\end{proof}

\begin{lem}\label{lem 3}
	Suppose $\mathrm{cf}(\alpha) \leq |\lambda| + \aleph_0$. Then for all $\xi < \lambda$,
	\[ \mathrm{cf}(\ell^\xi \alpha) \leq |\lambda| + \aleph_0 \]
\end{lem}
\begin{proof}
	The proof proceeds by induction on $\xi$ and for all ordinals $\alpha$ with $\mathrm{cf}(\alpha) \leq |\lambda| + \aleph_0$.
	The base case $\xi = 1$ is clear.
	
	Suppose $\xi = \sigma + 1$ is a successor ordinal and $\mathrm{cf}(\ell^\sigma \alpha) \leq |\lambda| + \aleph_0$.
	Let $\ell^\sigma \alpha = \eta + \omega^\beta$.
	
	If $\beta = 0$, then $\ell^{\sigma+1} \alpha = 0$.
	If $\beta = \theta + 1$, then $\ell^{\sigma+1} \alpha = \beta = \theta + 1$,
	and $\mathrm{cf}(\ell^{\sigma+1} \alpha) = 1$.
	If $\beta$ is limit, then $\ell^{\sigma+1} \alpha = \beta$ and
	$\mathrm{cf}(\ell^{\sigma+1} \alpha) = \mathrm{cf}(\beta)
		= \mathrm{cf}(\ell^\sigma \alpha) \leq |\lambda| + \aleph_0$.
		
	Now assume that $\xi$ is a limit ordinal. Suppose the induction holds for all $\sigma < \xi$. Write $\xi = \eta + \omega^\beta$. Then By Lemma~\ref{l lem}~\ref{l lem3} for all large $\sigma < \xi$
	\[ \ell^\sigma \alpha = e^{\omega^\beta} \ell^\xi \alpha, \]
	and then
	\[ \mathrm{cf}(\ell^\sigma \alpha) = \mathrm{cf}(e^{\omega^\beta} \ell^\xi \alpha) \leq |\lambda| + \aleph_0. \]
	 If $\ell^\xi \alpha$ is a limit ordinal,
	\[ \mathrm{cf}(\ell^\xi \alpha)=\mathrm{cf}(e^{\omega^\beta} \ell^\xi \alpha) = \mathrm{cf}(\ell^\sigma \alpha), \]
	and we are done by the induction hypothesis. Otherwise  $\ell^\xi \alpha$ is a successor ordinal and $\mathrm{cf}(\ell^\xi \alpha)=1 \leq |\lambda| + \aleph_0.$
\end{proof}

Now, we show that for a given ordinal $\lambda$, $\mathsf{GL}(\mathbb{P})$ is not strongly complete with respect to any $(\Theta,\mathcal{I}_\lambda)$ when $\mathbb{P}$ has sufficiently large cardinality. This establishes the first part of Theorem~\ref{incom uncount1}.

\begin{prop}
	Let $\lambda$ be any ordinal and $\kappa = |\lambda| +\aleph_0$.
	Let $\mathbb{P}=\{p_i\;: \; i<  (2^\kappa)^+   \}$ be a set of propositional variables with $|\mathbb{P}|= (2^\kappa)^+$.
	Let \[ \Gamma = \{ \Diamond p_i \; : \; i < (2^\kappa)^+ \} \cup \{ \Box(p_i \to \Diamond p_j) \; : \; i < j < (2^\kappa)^+ \}. \]
	Then $\Gamma$ is not satisfiable on any $(\Theta,\mathcal{I}_\lambda)$.
\end{prop}
\begin{proof}
	Suppose that $\mathfrak{M},\alpha \models \Gamma$ for some model $\mathfrak{M}=(\Theta_\lambda,\lbr\cdot\rbr)$ with $\alpha\in\Theta$.
	Let $\chi>(2^\kappa)^+ $ be a large enough regular cardinal and let $\lhd$ be a well-ordering of $\mathcal{H}(\chi),$ the collection of sets of hereditarily cardinality less than $\chi$.

We conclude a contradiction in two cases.
	\begin{description}
		\item[Case 1] Assume $\mathrm{cf}(\alpha) \leq |\lambda| + \aleph_0$.
		Then by Lemma \ref{lem 3}, for all $\xi < \lambda$,
		\[ \mathrm{cf}(\ell^\xi \alpha) \leq |\lambda| + \aleph_0 = \kappa. \]
		So for each such $\xi$ let
		\[ \langle \alpha_\xi(\eta)\; : \;\eta < \mathrm{cf}(\ell^\xi \alpha) \rangle \]
		be an increasing and cofinal sequence in $\ell^\xi \alpha$.	
		Define $\mathbf{c} : [(2^\kappa)^+]^2 \to [\lambda \times \kappa]^{<\omega}$ by
		$$\mathbf{c}(i, j) = B(i, j),$$
		where $B(i, j)$ is the $\lhd$-least finite subset of $\lambda \times \kappa$ such that  $\mathfrak{M},\gamma \models p_i \to \Diamond p_j$, for all $\gamma\in (-1, \alpha) \cap \bigcap_{(\xi, \eta) \in B(i, j)} (\alpha_\xi(\eta), \ell^\xi \alpha]_\xi$.

		By the Erd\H{o}s--Rado theorem, there are $S \subseteq (2^\kappa)^+$ of size $\kappa^+$ and $B \in [\lambda \times \kappa]^{<\omega}$ such that
		\[ \forall i < j \in S \quad \mathbf{c}(i, j) = B. \]
		Let $O= (-1, \alpha] \cap\bigcap_{(\xi, \eta) \in B} (\alpha_\xi(\eta), \ell^\xi \alpha]_\xi$. Then $O$ forms a basic neighbourhood of $\alpha$.
		Now, we proceed as in the proof of Proposition \ref{incom countable}, to conclude a contradiction. Thus choose
		$i_0<i_1<i_2<\dots$ from $S$.
		Since $\mathfrak{M},\alpha\models \lozenge p_{i_0}$, there exists $\beta_0\in O-\{\alpha\}$ with $\mathfrak{M},\beta_0\models p_{i_0}$. Note that $\ell^\xi \beta_0 > \alpha_\xi(\eta),$ for all $(\xi, \eta) \in B.$
		Then $\mathfrak{M},\beta_0\models p_{i_0}\wedge \lozenge p_{i_1}$, so we can find $\beta_1\in (-1, \beta_0) \cap O$ with $\mathfrak{M},\beta_1\models p_{i_1}$.
		Repeating this argument, we obtain a strictly decreasing sequence $\langle \beta_m\; : \; m<\omega\rangle$ of ordinals, a contradiction.

		\item[Case 2] Assume that $\mathrm{cf}(\alpha)>|\lambda|+\aleph_0=\kappa$.
		Then by Lemma~\ref{cf >}, $\mathrm{cf}(\ell^\xi\alpha)>\kappa$ for all $\xi<\lambda$.
		Now define  $\mathbf{c} : [(2^\kappa)^+]^2 \to [\lambda]^{<\omega}$ by
		$$\mathbf{c}(i, j) = B(i, j),$$
		where $B(i, j)$ is the $\lhd$-least finite subset of $\lambda$ such that
 for all $\xi\in B$ there exists $\beta_\xi<\ell^\xi\alpha$ such that
		$\mathfrak{M},\gamma\models p_i\to\lozenge p_j$ for all $\gamma\in (-1,\alpha)\cap \bigcap_{\xi\in B}(\beta_\xi,\ell^\xi\alpha]_\xi$.
		Again by the Erd\H{o}s--Rado theorem, there are a finite set $B \subseteq \kappa$ and a set $S\subseteq (2^\kappa)^+$ of size $\kappa^+$ such that $\mathbf{c}\upharpoonright [S]^2=B$.
		Let $i_0< i_1<\dots< i_n <\dots$ be a countable sequence of elements of $S$.
		Then for all $n<m$, we have $\mathbf{c}(i_n,i_m)=B$.
		So, for each $\xi\in B$ and each $n<m$, there is $\beta_{\xi(n,m)} < \ell^\xi\alpha$, such that
		$$(-1,\alpha)\cap \bigcap_{\xi\in B}(\beta_{\xi(n,m)},\ell^\xi\alpha]_\xi\subseteq \lbr p_{i_n}\to\lozenge p_{i_m}\rbr.$$
		Now for each $\xi\in B$, let $\beta_\xi = \sup_{n,m<\omega}\beta_{\xi(n,m)}$.
		Since $\mathrm{cf}(\ell^\xi\alpha)>\kappa \geq \aleph_0$, we have $\beta_\xi<\ell^\xi\alpha$, for all $\xi\in B$.
		Now consider the open basic neighbourhood $O=(-1,\alpha]\cap \bigcap_{\xi\in B}(\beta_\xi,\ell^\xi\alpha]$ of $\alpha$. Note that $O-\{\alpha\}$ is non-empty, since $\mathfrak{M},\alpha\models\lozenge p_i$, for all $i$.
		Again by the similar argument as in the first case one can conclude a contradiction.
	\end{description}
	So $\Gamma$ cannot be satisfied in $\alpha$.
\end{proof}

\subsection{Incompleteness with ${\tau_c}_{+\lambda}$}

It is known that it is consistent with ZFC + ``there exists a Mahlo cardinal" that $\mathsf{GL}$ is not complete with respect to ordinals with club topology, see \cite{blass}.
However, it is proved that if $\Theta> \aleph_{e^\lambda\omega+1}$, then $\mathsf{GL}$ is strongly complete with respect to $(\Theta, {\tau_c}_{+\lambda})$, where  $\tau_c$ is the club topology over $\Theta$, \cite[Corrolary~6.15]{aguilera2017strong}.
In this part we show that $\mathsf{GL}(\mathbb{P})$ is not strongly complete with respect to any $(\Theta, {\tau_c}_{+\lambda})$ for $\mathbb{P}$ of sufficiently high cardinality, thereby proving the second part of Theorem~\ref{incom uncount1}.

\begin{prop}
	Let $\lambda>0$ be any ordinal and $\kappa = |\lambda| +\aleph_0$.
	Let $\mathbb{P}=\{p_i\; : \; i<  (2^\kappa)^+   \}$ be a set of propositional variables with $|\mathbb{P}|= (2^\kappa)^+$.
	Let \[ \Gamma = \{ \Diamond p_i \; : \; i < (2^\kappa)^+ \} \cup \{ \Box(p_i \to \Diamond p_j) \; : \; i < j < (2^\kappa)^+ \}. \]
	Then $\Gamma$ is not satisfiable on any $(\Theta, {\tau_c}_{+\lambda})$.
\end{prop}
\begin{proof}
	Assume not, so, for some $\Theta>0$, there is a model $\mathfrak{M}$ based on $(\Theta, {\tau_c}_{+\lambda})$ and $\alpha<\Theta$ such that $\mathfrak{M},\alpha\models\Gamma$. Let $\chi>(2^\kappa)^+ $ be a large enough regular cardinal and let $\lhd$ be a well-ordering of $\mathcal{H}(\chi).$
	
	First note that for any $\mu>0$ we have $\mathrm{cf}(\alpha)=\Omega_\mu$ if and only if $\rho_{\tau_c}(\alpha)=\mu$ (see Example \ref{exam ord}).
	Also, since $\alpha$ should be a limit point, we have $\mathrm{cf}(\alpha)>\aleph_0$.
	So let $\mathrm{cf}(\alpha)=\Omega_\mu$ for some $\mu>0$.
	
	For each $\xi<\lambda$, we have $\ell^\xi\rho_{\tau_c}(\alpha) =\ell^\xi\mu$.
	Let
	$$\mathcal{B}_{\leq\kappa}=\{0<\xi<\lambda \; : \;  \mathrm{cf}(\ell^\xi\mu)\leq \kappa\},$$
	and
	$$\mathcal{B}_{>\kappa}=\{0<\xi<\lambda \; : \;  \mathrm{cf}(\ell^\xi\mu)>\kappa\}.$$
	For each $\xi\in \mathcal{B}_{\leq \kappa}$ fix $\langle\alpha_\xi(\eta) \; : \; \eta<\mathrm{cf}(\ell^\xi\mu)\rangle$ as an increasing and cofinal sequence in $\ell^\xi\mu$.
	Define the colouring
	$\mathbf{c}:[(2
	^\kappa)^+]^2 \to [\mathcal{B}_{\leq\kappa}\times \kappa]^{<\omega} \times [\mathcal{B}_{>\kappa}]^{<\omega}$
	by
	$$\mathbf{c}(i,j)= (B_0(i, j),B_1(i, j)),$$ where $(B_0(i, j),B_1(i, j))$ is the $\lhd$-least pair such that  for some club $C\subseteq\alpha$, and some sequence
 $\langle \beta_\xi \; : \; \xi \in B_1(i, j) \rangle$ with $\beta_\xi < \ell^\xi\mu$,   for  each $\gamma\neq\alpha$ in
	$$ C\cap \bigcap_{(\xi,\eta)\in B_0(i, j)}(\alpha_\xi(\eta),\ell^\xi\mu]_\xi^{\tau_c} \cap  \bigcap_{\xi\in B_1(i, j)}(\beta_\xi,\ell^\xi\mu]_\xi^{\tau_c},$$
	we have $\mathfrak{M}, \gamma \models p_i\to \lozenge p_j$.
	
	By the Erd\H{o}s--Rado theorem, we can find $S\subseteq (2^\kappa)^+$ of size $\kappa^+$ and $(B_0,B_1)$ such that $\mathbf{c}(i,j)=(B_0, B_1)$ for all $i<j\in S$.
	The rest is gone as before. Fix some $i_0<i_1<\dots$ in $S$. For $n<m<\omega$ and $\xi\in B_1$ let $\beta_\xi(n,m) <\ell^\xi\mu$ and $C(n,m)\subseteq \alpha$ a club witness
	$\mathbf{c}(i_n,i_m)=(B_0,B_1)$.
	Let $\beta_\xi= \sup_{n<m<\omega}\beta_\xi(n,m)<\ell^\xi\mu$ and
	$C=\bigcap_{n<m<\omega}C(n,m)\subseteq \alpha$. Then $C$ is a club.
	Now for all $n<m<\omega$,
	$$ [0,\alpha) \cap C\cap \bigcap_{(\xi,\eta)\in B_0}(\alpha_\xi(\eta),\ell^\xi\mu]_\xi^{\tau_c} \cap \bigcap_{\xi\in B_1}(\beta_\xi,\ell^\xi\mu]_\xi^{\tau_c} \; \subseteq \; \lbr p_{i_n}\to \lozenge p_{i_m}\rbr.$$
	Then we get the desired contradiction as before.
\end{proof}

\section{Strong Completeness of $\mathsf{GL}(\mathbb{P}$)}\label{sec:bouquet}

In this section we introduce  natural classes of topological spaces that yield strong completeness of $\mathsf{GL}(\mathbb{P})$ and $\mathsf{GL}.3(\mathbb{P})$ for uncountable languages $\mathbb{P}$. This construction generalizes the notion of an $\omega$-bouquet, originally defined in \cite{aguilera2017strong} to establish the strong completeness of $\mathsf{GL}$.

\subsection{$\lambda$-Bouquets}

We define $\lambda$-bouquets, which are reminiscent to Kripke frames with a slightly different interpretation rule and ordering of the nodes. One loosen the requirement for a modal formula $\Box\varphi$ to be satisfied. The idea behind the original definition of $\omega$-bouquet in \cite{aguilera2017strong} is that if a node of the tree has limit rank, then it is enough to satisfy $\varphi\land \Box\varphi$ at all but finitely many immediate successors. Dually $\Diamond \varphi$ hods if $\varphi\lor\Diamond\varphi$ holds at infinitely many immediate successors.

We slightly step away from the original approach, and let the above mentioned satisfaction definition happen at \emph{any} node with infinitely many immediate successors. In the $\omega$-case this does not entail any practical differences, but in the uncountable it slightly simplifies the things. If one is to follow the original approach, then the same proofs works with only cosmetic changes.
We further extend the satisfaction relation to the case of $\lambda$-many immediate successors. We say that $\Box \varphi$ holds if $\varphi\land \Box\varphi$ holds in a \emph{co-bounded subset} (i.e. the complement of a bounded subset) of the node's immediate successors, dually $\Diamond \varphi$ holds if $\varphi\lor \Diamond\varphi$ holds in a cofinal set of the node's immediate successors. Note that if $\lambda$ is a singular cardinal, then being cofinal is not permutation invariant. To amend this we additionally fix an enumeration of the immediate successors of each node. Moreover, one can relativise the definition by taking some other than co-bounded filter $F_\lambda$ on $\lambda$.

\begin{dfn}\label{def:lambda-bouquet}
	Let $\lambda>\aleph_0$ be a cardinal.
	A tree $(T,<_T)$ is called \textit{$\lambda$-bouquet} if:
  \begin{itemize}
		\item it is converse well-founded;
		\item if $t=\root(T)$, then $|\Suc_T(t)|=\lambda$,
		\item for every $s\in T$ with $t <_T s$, $|\Suc_T(s)|<\lambda$.
	\end{itemize}
\end{dfn}

\begin{dfn}\label{lambda tree}
  Let $T$ be a $\lambda$-bouquet tree. For each $w\in T$, let $e_w:|\Suc_T(w)|\to\Suc_T(w)$ be some bijective enumeration, if there is now risk of confusion we simply write $w_i$ instead of $e_w(i)$. Let \emph{$\text{Card}$} denote the class of all infinite cardinals and $\fil = \langle F_\kappa\subset \mathcal{P}(\kappa):\kappa \in \text{Card} \cap \lambda^+\rangle$ be a sequence of filters on infinite cardinals below $\lambda^+$. We define a topology $\sigma_T^\fil$ on $T$ inductively as follows:
	\begin{description}
    \item[$(*)_0$] Let $\sigma^\fil_0$ be the upset topology on $T$.
		\item[$(*)_1$] Define $\sigma^\fil_1$ as the least topology on $T$ containing $\sigma^\fil_0$ and
		all sets of the form
		$$\{w\}\cup \bigcup_{i\in D} T_{w_i}$$
    where $w\in T$ and $|\Suc(w)|=\kappa\ge\aleph_0$, $D\in F_\kappa$, and
    $w_i=e_w(i)$ for all $i<\kappa$.
		\item[$(*)_2$] Let $\sigma^\fil_2$ be the least topology on $T$ containing $\sigma^\fil_1$ and all sets of the form
			$$\{w\}\cup \bigcup_{i\in D} A_i$$
      where $w\in T$ and $|\Suc(w)|=\kappa\ge\aleph_0$, $D\in F_\kappa$, and
		each $A_i\subseteq T_{w_i}$ is $\sigma^\fil_1$-open with $w_i\in A_i$.
		\item[$(*)_{n+1}$] Given $\sigma^\fil_{n}$, define $\sigma^\fil_{n+1}$ analogously to the construction of $\sigma^\fil_2$ from $\sigma^\fil_1$.
		\item[$(*)$] Finally, let $\sigma^\fil_T$ be the topology defined based on $\bigcup_{n<\omega}\sigma^\fil_n$.
	\end{description}
\end{dfn}

Note that the rank of each point $w$ remains unchanged throughout the construction; consequently, its rank with respect to $\sigma_T^\fil$ coincides with its rank with respect to $\sigma_0^\fil$.

\begin{dfn}
	A topological space $(T,\sigma)$ is a \textit{$(\fil,\lambda)$-bouquet space} if:
	\begin{enumerate}
		\item $(T,<_T)$ is a $\lambda$-bouquet tree, for some $<_T$,
		\item $\sigma=\sigma_T^\fil$.
	\end{enumerate}
  We say that $(T,\sigma)$ is a $\lambda$-bouquet, if $\fil$ consists of co-bounded filters, i.e.
  \[(\forall \kappa \in \text{Card} \cap \lambda^+)(\forall A\subset \kappa) \big(A\in F_\kappa \leftrightarrow \exists \alpha<\kappa, \kappa\setminus \alpha \subset A \big).\]

  We also occasionally identify $(T,<)$ or even $T$ with $(T,\sigma)$, if there is no risk of confusion.
\end{dfn}

Clearly, every $(\fil,\lambda)$-bouquet space $(T,\sigma)$ is scattered.

As mentioned above, if $\lambda=\aleph_0$, then $(T,\sigma)$ is slightly differs from the $\omega$-bouquet defined in \cite{aguilera2017strong}, however the difference has no practical consequences and all proofs in \cite{aguilera2017strong} work with our definition as well.

Although we gave a topological definition for $\lambda$-bouquets, one can construe them as a modification of Kripke frames. The following proposition states the similarity of these structures in a precise way.

\begin{prop}\label{prop:kripke-topology-eq}
  Let $(T,<)$ be a  $\lambda$-bouquet tree with root $r$ with the corresponding $(\fil,\lambda)$-bouquet $(T,\sigma)$ and let $\inter\cdot:\mathbb P \to \mathcal{P}(T)$ be a valuation,
  then for any formula $\theta$ and $w\in T$, letting $\model=((T,\sigma),\inter\cdot)$, we have $\model,w\models \theta$ if and only if:
  \begin{itemize}
    \item $\theta =  \Diamond \varphi$ and either
    \begin{itemize}
      \item $\Suc_T(w)$ is finite and $\model, u \models \varphi \lor \Diamond\varphi$
      for some immediate successor $u$ of $w$, or
    \item $|\Suc_T(w)| = \kappa \ge\aleph_0$ and 
      $\set{i<\kappa :  \model, e_w(i) \models \varphi \lor \Diamond\varphi}\in (F_\kappa)^+$;
    \end{itemize} or
    \item $\theta =  \Box \varphi$ and either
    \begin{itemize}
      \item $\Suc_T(w)$ is finite and $\model, u \models \varphi \land \Box\varphi$
      for all immediate successors $u$ of $w$, or
      \item $|\Suc_T(w)| = \kappa \ge\aleph_0$ and
      $\set{i<\kappa : \model, e_w(i) \models \varphi \land \Box\varphi}\in F_\kappa$;
    \end{itemize}
  \end{itemize}
  Where $e_w:|\Suc_T(w)|\to \Suc_T(w)$ is as in Definition~\ref{lambda tree}.
\end{prop}
\begin{proof}
  The proof is by simultaneous induction on the formulas complexity and the rank of $w$.
  Assume $\theta = \Diamond\varphi$ and $\mathfrak{M},w\models \theta$ for some $w\in T$. If $\Suc_T(w)$ is finite, then the claim is clear, so suppose that it is infinite and $|\Suc_T(w)|=\kappa$. For the sake of contradiction assume $\set{i<\kappa : e_w(i) \models \varphi \lor \Diamond\varphi} \cap C = \emptyset$ for some $C\in F_\kappa$, thus $C\subset \set{i<\kappa :  e_w(i) \models \neg\varphi \land \Box\neg\varphi}$. It follows that for each $i\in C$, $\mathfrak{M},e_w(i) \models \neg \varphi$ and there is a neighbourhood $U_i$ of $e_w(i)$ such that $U_i\setminus\set{e_w(i)}\subset \inter{\neg\varphi}$. Thus, $\bigcup_{i\in C} U_i$ is a punctured neighbourhood of $w$ contained in $\inter{\neg \varphi}$, thus $\mathfrak{M},w\models \Box\neg\varphi$, a contradiction.

  For the opposite implication, assume  $|\Suc_T(w)| = \kappa \ge\aleph_0$ and the set
      $A=\set{i<\kappa :  \model, e_w(i) \models \varphi \lor \Diamond\varphi}$ belongs to  $(F_\kappa)^+$. Fix an open neighbourhood $U$ of $w$.
       Without loss of generality, $U=\{w\}\cup \bigcup_{i\in D} B_i$
      where  $D\in F_\kappa$, and
		each $B_i\subseteq T_{e_w(i)}$ is  open with $e_w(i)\in B_i$.

      Set
        \[
    Y_\varphi=\set{i<\kappa : \mathfrak{M},e_w(i) \models \varphi },\quad Y_{\Diamond\varphi}=\set{i<\kappa : \mathfrak{M},e_w(i) \models \Diamond\varphi },
  \]
      so that $A=Y_\varphi \cup Y_{\Diamond\varphi}.$
      If $Y_\varphi\in (F_\kappa)^+$, then $Y_\varphi \cap D \ne\emptyset,$ in particular,
        $\set{i <\kappa : e_w(i) \in U \land e_w(i)\models \varphi}\ne\emptyset$ and so $U\setminus\{w\}\cap \inter{\varphi}\ne\emptyset$.
  If $Y_{\Diamond\varphi}\in (F_\kappa)^+$ then $Y_{\Diamond\varphi} \cap D \ne\emptyset.$ Fix any $i$ in the intersection.
  If $|\Suc(e_w(i))|$ is finite, then $B_i \cap \inter{\varphi} \ne \emptyset$ and we are done. Otherwise,
 by the induction hypothesis, $B_i \cap \inter{\varphi}$  is in $F^+_{\kappa_i'}$ where $\kappa_i'=|\Suc(e_w(i))|$, hence again
   $U\setminus\{w\}\cap \inter{\varphi}\ne\emptyset$, and so $\model,w\models \Diamond\varphi$. The claim for
   $\theta=\Box\varphi$ follows.
\end{proof}

\begin{dfn}
	We call a $(\mathcal F,\lambda)$-bouquet $T$  {\it linear} if each node of the underlying tree has zero, one or infinitely many immediate successors. If additionally each $F\in \mathcal F$ is a non-principal ultrafilter, we call such bouquet an \emph{ultralinear $\lambda$-bouquet}.
\end{dfn}

\begin{prop}
The logic $\gl.3$ is sound with respect to the class of ultralinear $\lambda$-bouquets.
\end{prop}
\begin{proof}
  Recall that $\gl.3$ is the minimal set of formulas containing $\gl$ and $.3=\Box(\Box\varphi\to\psi)\lor\Box(\Box\psi\land\psi\to \varphi)$.
  $\gl$ is obviously sound for the ultralinear bouquets, since a refinement of a scattered space is scattered. We have to show that $T\models\Box(\Box \varphi\to\psi)\lor\Box(\Box\psi\land\psi\to\varphi)$ whenever $T$ is an ultralinear $\lambda$-bouquet. From here on we reason by induction on the rank of the node $x\in T$. Let us pick arbitrary $T$ and $\inter\cdot : \mathbb P \to \mathcal{P}(T)$. Until the end of the proof we can, thus, write $x\models \varphi$ instead of $(T,\inter\cdot),x\models \varphi$ without any risk of confusion. The base case of induction is trivial, since if $x$ has rank $0$, then $x\models \Box\bot$. Assume now that for any $y\in T_x$, $y\models \Box(\Box \varphi\to\psi)\lor\Box(\Box\psi\land\psi\to\varphi)$.
  
  \begin{description}
  	\item[Case 1] $x$ has one immediate successor $y$, then by the induction hypothesis on $y$ we have the following cases:
  	\begin{enumerate}
  		\item $y\models \Box(\Box\varphi\to\psi)$ and $y\models \neg\Box \varphi$;
  		\item $y\models \Box(\Box\varphi\to\psi)$ and $y\models \Box\varphi$ and $y\models \psi$;
  		\item $y\models \Box(\Box\varphi\to\psi)$ and $y\models \Box\varphi$ and $y\models \neg\psi$;
  		\item $y\models \Box(\Box\psi\land\psi\to\varphi)$ and $y\models \neg\Box \psi\lor\neg\psi$;
  		\item $y\models \Box(\Box\psi\land\psi\to\varphi)$ and $y\models \Box\psi\land\psi$ and $y\models \varphi$;
  		\item $y\models \Box(\Box\psi\land\psi\to\varphi)$ and $y\models \Box\psi\land\psi$ and $y\models \neg\varphi$;   \end{enumerate}
  	One can see that in cases (1), (2), (6) $x\models \Box(\Box\varphi\to\psi)$ and otherwise $x\models \Box(\Box\psi\land\psi\to\varphi)$.
  	\item[Case 2] $x$ has $\kappa$-many successors $\set{y_i}_{i<\kappa}$ where $\kappa\le\lambda$. For any $y\in \Suc_T(x)$ the same case distinction holds. For each $n\in[1,6]$ let $A_i=\set {i<\kappa: \text{(n) holds for } y_i}$. Clearly, $\bigcup A_i=\kappa$, moreover there is at leas one $i$ such that $A_i\in U_\kappa$, it follows from the fact that $\kappa\in U_\kappa$ and the fact that for any filter $U$ if $A\cup B\in U$, then $A\in U$ or $B\in U$. Thus, for $U_\kappa$-almost all successors of $x$, one of the (1)-(6) holds.
  \end{description}
\end{proof}

\subsection{Strong Completeness Results}

In this section we establish strong completeness of $\gl(\mathbb P)$ with respect to $\lambda$-bouquets and of $\gl.3(\mathbb P)$ with respect to ultralinear $\lambda$-bouquets, where $\lambda=|\mathbb P|$. 
Whenever possible, we present proofs that apply to both $\gl$ and $\gl.3$. In such cases, we typically denote by $\mathsf{L}$ a logic from the set $\set{\gl,\gl.3}$.

Until the end of the section we fix an increasing and continuous sequence $\langle \lambda_i  \; : \; i<\mathrm{cf}(\lambda)\rangle$ of limit ordinals such that:
	\begin{itemize}
		\item $\lambda_i<\lambda$, for all $i< \mathrm{cf}(\lambda)$;
		\item $\sup_i \lambda_i=\lambda$;
		\item if $\lambda=\mu^+$ is a successor cardinal, then $|\lambda_{i+1}-\lambda_i|=\mu$ for all $i$;
		\item if $\lambda$ is a limit cardinal, then each $\lambda_i$ is a cardinal.
	\end{itemize}
  And let $\mathbb{P}_i =\{p_\xi \; : \;  \xi<\lambda_i\}$ for each $i<\cof(\lambda)$. Note that for each $i<\cof(\lambda)$, we have $|\mathbb{P}_i|= |\lambda_i|<\lambda$.
  \begin{dfn}\label{def:FGamma-sequence}
    Let $\Gamma\subset \lang(\mathbb P)$ be a maximal $\mathsf L$-consistent set of formulas with $|\Gamma|=\lambda$, and $\Diamond\top\in\Gamma$.
    Let $F$ be a filter on $\lambda$. We say that a sequence $\tup{\Delta_\xi: \xi <\lambda}$ of sets of formulas is an \emph{$(\mathsf L, F,\Gamma)$-sequence} if for each $\xi<\lambda$:
  \begin{enumerate}
    \item if $\xi<\lambda_i$, then $\Delta_\xi\subset\lang(\mathbb P_i)$;\label{item:FGamma-sequence-3}
    \item if $\xi<\lambda_i$, then $|\Delta_\xi|<\lambda_i$;\label{item:FGamma-sequence-1}
    \item $\Delta_\xi$ is $\mathsf L$-consistent for all $\xi<\cof(\lambda)$;\label{item:FGamma-sequence-2}
  \end{enumerate}
  and
  \begin{enumerate}
    \item[(4)] For each $\psi$ such that $\Diamond \psi\in \Gamma$, we have $\set{\xi<\lambda : \psi \in \Delta_\xi}\in F^+$;\namedlabel{4}{item:FGamma-sequence-4}
    \item[(5)] For each $\varphi$ such that $\Box\varphi\in \Gamma$, we have $\set{\xi<\lambda : \Box\varphi,\varphi \in \Delta_\xi}\in F$.\namedlabel{5}{item:FGamma-sequence-5}
  \end{enumerate}
\end{dfn}

\begin{lem}\label{lemma:lottery}
  Let $\Gamma$ be an $\mathsf{L}$-consistent set of formulas of cardinality $\lambda$ such that $\Diamond\top\in\Gamma$. Let $\tup{\Delta_\xi:\xi<\lambda}$ be an $(\mathsf L, \fil,\Gamma)$-sequence. Assume that for each $\xi<\lambda$, there are:
  \begin{itemize}
    \item a $\lambda_i$-bouquet tree $(T_\xi,<_\xi)$ with the corresponding $(\fil,\lambda_i)$-bouquet space $(T,\sigma_\xi)$, for some $\lambda_i>\xi$;
    \item a model $\model_\xi=(T_\xi,\sigma_\xi,\inter\cdot_\xi)$ such that $\model_\xi,t_\xi\models \Delta_\xi$, where $t_\xi=\root(T_\xi)$;
  \end{itemize}

  Then there is a model $\model=(T,\sigma,\inter\cdot)$, where $(T,\sigma)$ is $(\fil,\lambda)$-bouquet space, such that $\model,t\models \Gamma$, where $t=\root(T)$. Moreover, if all $T_\xi$'s are linear bouquet trees, then so is $T$.
\end{lem}
\begin{proof}

  Let $T=\bigoplus_{\xi<\lambda}T_\xi$ be the lottery sum of $T_\xi$'s which is defined as follows:
	\begin{itemize}
		\item $root(T)=t$, for some new $t$ which is not in any $T_\xi$,
		\item $Suc_T(t) = \{t_\xi \; : \;  \xi<\lambda\}$,
		\item $T_{t_\xi} =T_\xi$, for all $\xi<\lambda$.
	\end{itemize}
  Clearly, $T$ is a $\lambda$-bouquet tree and if all $T_\xi$' are linear bouquet trees, then so is $T$.
	Let $(T,\sigma)$ be the associated $\lambda$-bouquet space of $T$, and define the model
	$\mathfrak{M}=(T,\sigma,\lbr\cdot\rbr)$, where the valuation is given by
	$\lbr p\rbr=\lbr p\rbr_*\cup \bigcup_{\xi<\lambda}\lbr p\rbr_\xi$
  for each $p\in\mathbb{P}$, with $\lbr p\rbr_*=t$ iff $p\in\Gamma$.

	We show now that $\mathfrak{M},t\models\Gamma$.
  The proof proceeds by induction on the complexity of formulas. The base case and Boolean connectives are straightforward, whereas the modal cases follow from Proposition~\ref{prop:kripke-topology-eq} and \eqref{item:FGamma-sequence-4} and \eqref{item:FGamma-sequence-5} of Definition~\ref{def:FGamma-sequence}.
\end{proof}

Until the rest of the section we fix $\fil = \set{F_\kappa:\kappa\in\mathrm{Card}}$ to be the class of the co-bounded filters and $\mathcal U = \set{U_\kappa:\kappa\in \mathrm{Card}}$ to be a class of non-principal ultrafilters with $U_\kappa\supset F_\kappa$.

\begin{lem}\label{lemma:gl-FGammma}
  For any maximal $\gl$-consistent set of $\lang(\mathbb P)$-formulas $\Gamma$, if $\Diamond\top\in\Gamma$, then there is a $(\gl,F_\lambda,\Gamma)$-sequence.
\end{lem}
\begin{proof}
  Let $\{\varphi_\xi \; : \;  \xi<\lambda\}$ and
	$\{\psi_\xi \; : \;  \xi<\lambda\}$ be a list of $\mathcal{L}(\mathbb{P})$ formulas such that for each $i< \mathrm{cf}(\lambda)$ the following conditions hold:
	\begin{itemize}
		\item $\{\psi_\xi \; : \;  \xi<\lambda\}$ lists all $\mathcal{L}(\mathbb{P})$-formulas $\psi$ such that $\lozenge\psi\in\Gamma$;
		\item if $\psi\in\mathcal{L}(\mathbb{P}_i)$, $i< \mathrm{cf}(\lambda)$ and $\lozenge\psi\in\Gamma$, then
		$\{\xi<\lambda_i \; : \;  \psi=\psi_i\}$ is unbounded in $\lambda_i$;
		\item $\{\varphi_\xi \; : \;  \xi<\lambda\}$ lists all formulas in $\mathcal{L}(\mathbb{P})$ with $\Box\varphi\in\Gamma$;
		\item if $\varphi\in \mathcal{L}(\mathbb{P}_i)$ and $\Box\varphi\in\Gamma$, then $\varphi=\varphi_\xi$ for a unique $\xi<\lambda_i$.
	\end{itemize}
  For each $\xi<\lambda$ set
	$$\Delta_\xi= \{\psi_\xi\} \cup \{\varphi_\zeta\wedge \Box\varphi_\zeta \; : \;  \zeta<\xi\}.$$
  Trivially, \eqref{item:FGamma-sequence-1} of Definition~\ref{def:FGamma-sequence} holds. Note that:
	\begin{description}
		\item[(a)] If $\xi<\lambda_i$, then $\Delta_\xi\subseteq \mathcal{L}(\mathbb{P}_i)$.
		\item[(b)] $\Delta_\xi$ is consistent. Indeed, for any finite subset $\Delta' =\{\psi_\xi\} \cup \{\varphi_{\zeta_i}\wedge \Box\varphi_{\zeta_i} \; : \;  i<n\}\subseteq\Delta_\xi$, we have $\Gamma\vdash\lozenge\bigwedge\Delta'$. Since $\Gamma$ is consistent, so is  $\Delta'$.
	\end{description}
  Thus, \eqref{item:FGamma-sequence-2} and \eqref{item:FGamma-sequence-3} hold. It is also clear from construction that \eqref{item:FGamma-sequence-4} and \eqref{item:FGamma-sequence-5} hold.
\end{proof}

\begin{lem}\label{lemma:gl3-FGammma}
  For any maximal $\gl.3$-consistent set of $\lang(\mathbb P)$-formulas $\Gamma$, if $\Diamond\top\in \Gamma$, then  there is a $(\gl.3, U_\lambda,\Gamma)$-sequence.
\end{lem}
\begin{proof}
  Let $\{\varphi_\xi \; : \;  \xi<\lambda\}$ and
	$\{\psi_\xi \; : \;  \xi<\lambda\}$ be a list of $\mathcal{L}(\mathbb{P})$ formulas such that for each $i< \mathrm{cf}(\lambda)$ the following conditions hold:
	\begin{itemize}
		\item $\{\psi_\xi \; : \;  \xi<\lambda\}$ lists all $\mathcal{L}(\mathbb{P})$-formulas $\psi$ such that $\lozenge\psi\in\Gamma$;
		\item if $\psi\in \mathcal{L}(\mathbb{P}_i)$ and $\Diamond\psi\in\Gamma$, then $\psi=\psi_\xi$ for a unique $\xi<\lambda_i$.
		\item $\{\varphi_\xi \; : \;  \xi<\lambda\}$ lists all formulas in $\mathcal{L}(\mathbb{P})$ with $\Box\varphi\in\Gamma$;
		\item if $\varphi\in \mathcal{L}(\mathbb{P}_i)$ and $\Box\varphi\in\Gamma$, then $\varphi=\varphi_\xi$ for a unique $\xi<\lambda_i$.
	\end{itemize}
  For each $\xi<\lambda$ set
	$$\Delta_\xi= \{\psi_\zeta \lor \Diamond\psi_\zeta:\zeta<\xi\} \cup \{\varphi_\zeta\wedge \Box\varphi_\zeta \; : \;  \zeta<\xi\}.$$

  Exactly as in the previous lemma, \eqref{item:FGamma-sequence-1} and \eqref{item:FGamma-sequence-2} hold. As of \eqref{item:FGamma-sequence-4} and \eqref{item:FGamma-sequence-5}, they follow from the assumption that $U$ extends the co-bounded filter. It is left to verify \eqref{item:FGamma-sequence-3} by showing that $\Delta_\xi$ is consistent for each $\xi<\lambda$. Assume otherwise, then there are finite $I, J\subset \xi$ such that:
    \[
      \gl.3\vdash \bigwedge_{i\in I} \varphi_i\land\Box\varphi_i \to \bigvee_{j\in J} \neg\psi_j\land\Box\neg\psi_j.
  \]
  Applying normality and transitivity, we get
   \[
      \gl.3\vdash \bigwedge_{i\in I} \Box\varphi_i \to \Box\bigvee_{j\in J} \neg\psi_j\land\Box\neg\psi_j.
  \]
  Since, $\bigwedge_{i\in I}\Box\varphi_i\in \Gamma$, we have $\Box\bigvee_{j\in J} \neg\psi_j\land\Box\neg\psi_j\in \Gamma$ as well.
  Thus,
  \begin{equation}\label{eq:inconsistent-formula}\tag{$\mathrel\#$}
  \bigwedge_{j\in J}\Diamond \psi_j\land\Box\bigvee_{j\in J} \neg\psi_j\land\Box\neg\psi_j \in \Gamma.
\end{equation}
We argue that \eqref{eq:inconsistent-formula} is inconsistent with $\gl.3$. To derive contradiction, we show that in $\gl.3$, $\Box((p\land \Box p)\lor (q\land \Box q))$ implies $\Box p$ or $\Box q$.  \[
    \gl.3\vdash \Box(\Box p\to  q)\lor \Box(\Box q \land q \to p),
  \]
  \[
    \gl.3\vdash \Box(\Box p\land p\to  q)\lor \Box(\Box q \land q \to p).
  \]
  Thus,
  \[
    \gl.3 \vdash  \Box((\Box p\land p)\lor(\Box q\land q))\to \Box p\lor \Box q.
  \]
  It follows that \eqref{eq:inconsistent-formula} is inconsistent with $\gl.3$.
\end{proof}

We now state our main result.

\begin{thm*}[Theorem~\ref{mainthm}]	For any $\lambda\geq \aleph_0$,
  \begin{enumerate}
      \item the logic $\mathsf{GL}(\mathbb{P})$ is strongly complete with respect to $\lambda$-bouquet spaces, where $|\mathbb{P}|=\lambda$;\label{item:main-theorem-1}
      \item the logic $\mathsf{GL.3}(\mathbb{P})$ is strongly complete with respect to ultralinear $\lambda$-bouquet spaces, where $|\mathbb{P}|=\lambda$;\label{item:main-theorem-2}
    \end{enumerate}
\end{thm*}
\begin{proof}
  Let $\mathsf L$ be either $\gl$ or $\gl.3$. Suppose that $\Gamma$ is a maximal $\mathsf L$-consistent set of $\mathsf{L}(\mathbb{P})$-formulas.
  First, note that if $\Box\bot\in \Gamma$, then $\Gamma$ can be satisfied at a model with $T=\set{t}$. For the rest of the proof, we assume that $\Diamond \top \in \Gamma$.
  We prove the result by induction on $\lambda$.

  For $\lambda=\aleph_0$ and $\gl$, the claim is already established in \cite[Theorem~4.9]{aguilera2017strong}. However, since our definition of bouquets slightly diverges from the original approach, we take $\lambda<\aleph_0$ as the base case, which is essentially the Kripke completeness for finite set of formulas.

  Assume now that $\lambda\ge\aleph_0$, and the theorem holds for all cardinals strictly less than $\lambda$.
  If $\mathsf L =\gl$, then by Lemma~\ref{lemma:gl-FGammma} there is an $(\gl,F_\lambda,\Gamma)$-sequence $\tup{\Delta_\xi:\xi<\lambda}$.
  If $\mathsf L =\gl.3$, then by Lemma~\ref{lemma:gl3-FGammma} there is an $(\gl.3, U_\lambda,\Gamma)$-sequence $\tup{\Delta_\xi:\xi<\lambda}$. By the induction hypothesis and the fact that $|\Delta_\xi|<\lambda$ and $\Delta_\xi$ is consistent, there are models $\model_\xi=(T_\xi,\sigma_\xi,\inter\cdot_\xi)$ such that $\model_\xi,t_\xi\models \Delta_\xi$, where $t_\xi=\root(T_\xi)$ for every $\xi<\lambda$, then by Lemma~\ref{lemma:lottery}, there is a model $\model=(T,\sigma,\inter\cdot)$, such that $\model,t\models \Gamma$, where $t=\root(T)$. Note that Lemma~\ref{lemma:lottery} guarantees that $T$ is a linear bouquet tree, if $T_\xi$ were so.
\end{proof}

Note that the existence of a non-principle ultrafilter is independent of $\zf$. In fact, the result for $\gl.3(\mathbb P)$ with $|\mathbb P|=\aleph_0$ need not hold in a choiceless setting, as opposed to the case of $\gl(\mathbb P)$, where the proof can be carried out in a weaker setting.

\begin{prop}\label{prop:not-uf}
  It is consistent with $\zf$, that $\gl.3(\mathbb P)$ is not strongly complete with respect to the class of linear $(\set{U},\omega)$-bouquets for any filter $U$ on $\omega$.
\end{prop}
\begin{proof}
  It is consistent with $\zf$ that there is no non-principal ultrafilters on $\omega$, thence $U$ is either principle or not an ultrafilter. If the ultrafilter is principle, then it is closed under arbitrary intersections, so, in particular, the correspondent topology must be $G_\delta$. By Proposition~\ref{g-delta}, such topologies fail to yield the strong completeness.

  Whereas if $U$ is not an ultrafilter, then there is a $B\subset \omega$ such that $B\notin U$ and $\omega\setminus B\notin U$, this implies $B\in U^+$ and $\omega\setminus B\in U^+$ (see Def.~\ref{def:filter}). We now define a model $(T,<,\lbr\cdot\rbr)$ as follows. Let $T = \set{t}\cup\set{a_i : i <\omega}$ and  for each $x,y\in T$, we put $x<y$ whenever $x=t$ and $y = a_i$ for some $i$ and let $\lbr\cdot\rbr$ be such that $\lbr p\rbr = \set{a_i : i\in B}$. Then, by Proposition~\ref{prop:kripke-topology-eq},
 $T,t\models \Box\Box\bot \land \Diamond p \land \Diamond \neg p$, which is inconsistent with $\gl.3$.
\end{proof}

\section{Concluding Remarks}
Although it is proved that $\mathsf{GL}$ is strongly complete with respect to any ordinal $(\Theta,\mathcal{I}_\lambda)$ with $\Theta > e^\lambda\omega$ and $(\Theta, {\tau_c}_{+\lambda})$ with $\Theta > \aleph_{e^\lambda\omega+1}$ , we show that for each $\lambda$ we can find a set of propositional variables $\mathbb{P}$ of uncountable size that refutes the strong completeness of $\mathsf{GL}(\mathbb{P})$.
Specifically,  we showed that for any given $\lambda$, $\mathsf{GL}(\mathbb{P})$ is not strongly complete with respect to any $(\Theta,\mathcal{I}_\lambda)$  or $(\Theta, {\tau_c}_{+\lambda})$, provided that $|\mathbb{P}|\geq(2^\kappa)^+$ where $\kappa=|\lambda|+\aleph_0$.
There are some questions that remain unsolved,
for instance:
\begin{question}
Suppose $|\mathbb{P}|=\aleph_1$. Is  $\mathsf{GL}(\mathbb{P})$  strongly complete with respect to some ordinal space, or with respect to some
generalized Icard space $\Theta_\lambda$, for some $\lambda>0?$
\end{question}
More generally, can we find a counterexample for a set of formulas in a language of lower cardinality?

\subsection*{Acknowledgment}
The first author's research has been supported by a grant from IPM (No. 1405030417). The first and third authors work is based upon research funded by Iran National Science Foundation (INSF) under project No. 40401358. They are also grateful to Massoud Pourmahdian for his insightful comments on an earlier draft of this paper.

The second author's work has been partially supported by the Austrian Science Foundation (FWF) through grants 10.55776/STA139 and 10.55776/P36837. The second author is grateful to Juan Aguilera and Thibaut Kouptchinsky for useful comments and insights, and to the scientific and  organizing committee of the conference ``PhDs in Logic 2023'', where the author presented their early results on this topic.

\end{document}